\documentclass[12pt,leqno]{amsart}  
\usepackage{amsmath,amstext,amsthm,amssymb,amsxtra}
\usepackage[colorlinks,citecolor=red,pagebackref,hypertexnames=true]{hyperref}
\usepackage[sorted,author-year,backrefs,msc-links,nobysame,initials]{amsrefs}
\usepackage{todonotes}
\usepackage{txfonts} %also txfonts 
\usepackage[T1]{fontenc}
\usepackage{lmodern}

\usepackage{mathtools,MnSymbol}
\mathtoolsset{showonlyrefs,showmanualtags}

\allowdisplaybreaks

\theoremstyle{plain} % definition 
\newtheorem{lemma}[equation]{Lemma} 
\newtheorem{proposition}[equation]{Proposition} 
\newtheorem{theorem}[equation]{Theorem} 
\newtheorem{corollary}[equation]{Corollary} 

\theoremstyle{definition}
\newtheorem{definition}[equation]{Definition} 

\theoremstyle{remark}
\newtheorem{remark}[equation]{Remark}
\newtheorem{example}[equation]{Example} 

\newtheorem{question}[equation]{Question}

\numberwithin{equation}{section}

\def\norm#1.#2.{\lVert#1\rVert_{#2}}
\def\Norm#1.#2.{\bigl\lVert#1\bigr\rVert_{#2}}
\def\NOrm#1.#2.{\Bigl\lVert#1\Bigr\rVert_{#2}}
\def\NORm#1.#2.{\biggl\lVert#1\biggr\rVert_{#2}}
\def\NORM#1.#2.{\Biggl\lVert#1\Biggr\rVert_{#2}}

\def\ip#1,#2,{\langle #1,#2\rangle}
\def\Ip#1,#2,{\bigl\langle#1,#2\bigr\rangle}
\def\IP#1,#2,{\Bigl\langle#1,#2\Bigr\rangle}

\def\ABs#1{\biggl\lvert#1\biggr\rvert}

\def\XXint#1#2#3{{\setbox0=\hbox{$#1{#2#3}{\int}$}
     \vcenter{\hbox{$#2#3$}}\kern-.5\wd0}}

\def\eqdef{\stackrel{\mathrm{def}}{{}={}}}
\def\dup{\overline{\textnormal d}}
\def\dlo{\underline{\textnormal d}}

\begin{document}

\title[Topological multiple recurrence]{Szemer\'edi's theorem, frequent hypercyclicity and multiple recurrence}

\author[G. Costakis]{George Costakis}
\address{Department of Mathematics, University of Crete, Knossos Avenue, GR-714 09 Heraklion, Crete, GREECE.} \email{costakis@math.uoc.gr}

\author[I. Parissis]{Ioannis Parissis$^{(1)}$}
\thanks{$^{(1)}$Research partially supported by Funda\c{c}\~{a}o para a Ci\^{e}ncia e a Tecnologia (FCT/Portugal) through MATH-LVT-Lisboa-89.}
\address{Centro de An\'{a}lise Matem\'{a}tica, Geometria e Sistemas Din\^{a}micos, Departamento de Matem\'atica, Instituto Superior T\'ecnico, Av. Rovisco Pais, 1049-001 Lisboa, PORTUGAL.}
\email{ioannis.parissis@gmail.com}


\begin{abstract} Let $T$ be a bounded linear operator acting on a complex Banach space $X$ and $(\lambda_n)_{n\in\mathbb N}$ a sequence of complex numbers. Our main result is that if $|\lambda_n|/|\lambda_{n+1}|\to 1$ and the sequence $(\lambda_n T^n)_{n\in\mathbb N}$ is frequently universal then $T$ is topologically multiply recurrent. To achieve such a result one has to carefully apply Szemer\'edi's theorem in arithmetic progressions. We show that the previous assumption on the sequence $(  \lambda_n)_{n\in\mathbb N}$ is optimal among sequences such that $|\lambda_{n}|/|\lambda_{n+1}|$ converges in $[0,\infty]$. In the case of bilateral weighted shifts and adjoints of multiplication operators we provide characterizations of topological multiple recurrence in terms of the weight sequence and the symbol of the
multiplication operator respectively.
\end{abstract}

\maketitle

\section{Introduction}
     In this note we discuss how some notions in dynamics of linear operators are connected to classical notions in
     topological dynamics using in an essential way Szemer\'edi's theorem in arithmetic progressions. Let us first fix
     some notation. As usual the symbols $\mathbb{N}$, $\mathbb{Z}$ stand for the sets of positive integers and integers respectively. Throughout this paper the
     letter $X$ will denote an infinite dimensional separable Banach space over the field of complex numbers
     $\mathbb{C}$. We denote by $\mathbb{D}$ the open unit disk of the complex plane centered at the origin and by $\mathbb{T}$ the unit circle.
     For $x\in X$ and $r$ a positive number we denote by $B(x,r)$ the open ball with center $x$ and radius $r$,
     i.e. $B(x,r)=\{ y\in X: \|  y-x\| <r\}$. For a subset $D$ of $X$ the symbol $\overline{D}$ denotes the closure of $D$.
     In general $T$ will be a bounded linear operator acting on $X$. For simplicity we will refer
     to $T$ as an operator on $X$.

     \begin{definition}
     The operator $T$ is called
     \emph{hypercyclic} if there exists a vector $x\in X$ whose orbit under $T$, i.e. the set
     $$\textnormal{Orb}(T,x)\eqdef\{ T^nx:n=0,1,2,\ldots \} ,$$
     is dense in $X$ and in this case $x$ is called a \emph{hypercyclic} vector for $T$.
     \end{definition}

     Under our assumptions on $X$,
     it is easy to check that hypercyclicity is equivalent to the notion of topological transitivity, i.e. $T$ is
     \emph{topologically transitive} if for every pair $(U,V)$ of non-empty open sets of $X$ there exists a positive
     integer $n$ such that $T^nU\cap V\neq \emptyset$. Hypercyclicity is a phenomenon which occurs only in infinite
     dimensions.
     For several examples of hypercyclic operators and a thorough analysis of linear dynamics,
     we refer to the recent books \cite{BM}, \cite{GrPe}.

     A more recent notion relevant to hypercyclicity, which examines how often the orbit of a hypercyclic
     vector visits a non-empty open set, was introduced by Bayart and Grivaux in \cite{BaGr1}, \cite{BaGr2}:
     \begin{definition}
     $T$ is called \emph{frequently hypercyclic}  if there is a vector $x$ such that for every non-empty open
     set $U$ the set
     $$\{ n\in \mathbb{N}: T^nx \in  U\}   $$
     has positive lower density.
     \end{definition}
     Recall that the lower and upper densities of a subset $B$ of $\mathbb{N}$ are
     defined as
     $$ \dlo(B)\eqdef\liminf_{N\to \infty}\frac{ |\{ n\in B :n\leq N \} |}{N}, \quad
        \dup(B)\eqdef\limsup_{N\to \infty}\frac{ |\{ n\in B :n\leq N \} |}{N} ,$$
     respectively. Here $|C|$ denotes the cardinality of a set $C\subset \mathbb{N}$.

Actually one can define a more general notion of frequent hypercyclicity for sequences of operators and in fact this notion has been introduced in \cite{BG} as follows:

\begin{definition} Let $(T_n)_{n\in\mathbb N}$ be a sequence of operators. We say that $(T_n)_{n\in\mathbb N}$ is \emph{frequently universal} if there is a vector $x\in X$ such that for every open set $U$ the set
	$$\{n\in\mathbb N: T_n x\in U\}$$
has positive lower density. 
\end{definition}

In this note our purpose is to connect the previous notions from linear dynamics to classical notions from topological dynamics and, in particular, to that of recurrence and (topological) multiple recurrence; see \cite{Fur}. We recall these notions here.

\begin{definition}
An operator $T$ is called \emph{recurrent} if for every non-empty open set $U$ in $X$ there is some positive integer $k$ such that
$$U\cap T^{-k}U \neq \emptyset .$$ A vector $x\in X$ is called \emph{recurrent for $T$} if there exists an increasing
sequence of positive integers $\{ n_k\}$ such that $T^{n_k}x\to x$ as $k\to +\infty$.
\end{definition}
\begin{remark} \label{r.rec}
It is easy to see that if $T$ is a recurrent operator on $X$ then the set of recurrent vectors for $T$ is dense
in $X$. Moreover, if $T$ is an invertible operator then $T$ is recurrent if and only if its inverse $T^{-1}$ is
recurrent.
\end{remark}

\begin{definition}
An operator $T$ is called \emph{topologically multiply recurrent} if for every non-empty open set $U$ in $X$ and
every positive integer $m$ there is some positive integer $k$ such that
$$U\cap T^{-k}U \cap \ldots \cap T^{-mk}U \neq \emptyset .$$
\end{definition}
Clearly every hypercyclic operator is recurrent. Of course, there is no reason for a hypercyclic operator
to be topologically multiply recurrent in general. On the other hand, as we remark in Section \ref{ss.Aimplies}, one can
trivially deduce from Szemer\'edi's theorem in arithmetic progressions that a frequently hypercyclic operator is
in fact topologically multiply recurrent.

In \cite{CoRu} it was shown that $T$ is hypercyclic whenever $(\frac{1}{n} T^n)_{n\in\mathbb N}$ is frequently universal. Motivated by this, we examine when the hypothesis that $(\lambda_nT^n)_{n\in\mathbb N}$ is frequently universal for some  sequence of complex numbers $(\lambda_n)_{n\in \mathbb N}$ implies that $T$ is topologically multiply recurrent. Our main result is the following:

\begin{theorem}\label{t.main1} Let $( \lambda_n)_{n\in\mathbb N}$ be a sequence of complex numbers such that 
	$$ \lim_{n\to\infty}\frac{|\lambda_n|}{ |\lambda_{n+\tau}|} =1,$$
for some positive integer $\tau$. Suppose that the sequence of operators $(\lambda_n T^n)_{n\in\mathbb N}$ is frequently universal. Then $T$ is topologically multiply recurrent.
\end{theorem}

The hypothesis of Theorem \ref{t.main1} is optimal among complex sequences such that the limit $\lim_{n\to\infty}|\lambda_n|/|\lambda_{n+\tau}|$ exists. More precisely we have:

\begin{proposition}\label{p.bad} Let $a\in\mathbb [0,\infty)\setminus\{1\}$  and $\tau$ be a positive integer. There exists a sequence $( \lambda_n)_{n\in\mathbb N}$ and an operator $T$ which is not even recurrent, such that  
	$$ \lim_{n\to\infty}\frac{|\lambda_n|} {|\lambda_{n+\tau}|}=a $$
 and the sequence $( \lambda_nT^n)_{n\in\mathbb N}$ is frequently universal.
\end{proposition}

The rest of the paper is organized as follows. In paragraph \ref{ss.szem} we state Szemer\'edi's theorem and describe its variations which result from the ergodic-theoretic proofs. In paragraph \ref{ss.weakerproperties} we describe some weaker properties that follow from frequent universality. We will use these weaker notions to prove the main theorem in paragraph \ref{ss.Aimplies}. We close this section by giving some forms of polynomial recurrence that follow under the same hypotheses as in Theorem \ref{t.main1}. In section \ref{s.bad} we give the proof of Proposition \ref{p.bad}. We also provide several examples of sequences such that our main theorem applies. In section \ref{s.shifts} we characterize topologically multiply recurrent weighted shifts in terms of their weight sequences. In section \ref{s.adjoints} we look at adjoints of
multiplication operators on suitable Hilbert spaces. We characterize when these operators are topologically
multiply recurrent by means of a geometric condition on the symbol of the multiplication operator. It turns out
that such operators are topologically multiply recurrent if and only if they are frequently hypercyclic. Finally
in section \ref{s.remarks} we propose some further questions related to the results of the present paper.

\section{Acknowledgments} Part of this work was carried out during the second author's visit at the Department of Mathematics of the University of Crete in August 2010. The second author is especially grateful to Mihail N. Kolountzakis for his hospitality during this visit. We are grateful to the referee for a very careful reading of the original manuscript. We would also like to thank Nikos Frantzikinakis for bringing to our attention the strengthened form of Szemer\'edi's theorem.

%%%%%%%%%%%%%%%%%%%%%%%%%%%%% SECTION SECTION SECTION
%%%%%%%%%%%%%%%%%%%%%%%%%%%%% SECTION SECTION SECTION
\section{Szemer\'edi's theorem and multiple recurrence} \label{s.main}

In this section we give the proof of Theorem \ref{t.main1}. Our proof relies heavily on Szemere\'edi's theorem in arithmetic progressions and variations of it so we first recall its statement and explain some of its generalizations.

%%%%%%%%%%%%%%%%%%%%%%%%%%%%% SUBSECTION  SUBSECTION SUBSECTION
%%%%%%%%%%%%%%%%%%%%%%%%%%%%% SUBSECTION  SUBSECTION SUBSECTION
\subsection{Szemer\'edi's theorem and variations}\label{ss.szem} The `infinite' version of the classical Szemer\'edi theorem on arithmetic progressions is usually stated in the following form:

\begin{theorem}[Szemer\'edi's theorem]\label{t.szem1} Let $A\subset \mathbb N$ have positive upper density $\dup(A)>0$. Then for any positive integer $m$ there exist positive integers $a,r$ such that the $m$-term arithmetic progression
	$$a,a+r,a+2r,\ldots,a+mr,$$
is contained in $A$.
\end{theorem}

By now there are numerous proofs of this theorem as well as several extensions of it. We will not try to give an account of those here. Instead, we refer the interested reader to \cite{TV} and the references therein. Here we are interested in the following strengthened version which is implicit in any ergodic theoretic proof of Szemer\'edi's theorem and, in particular, in \cite{Fur1,FurKa}.

\begin{theorem}\label{t.szem2} Let $A\subset \mathbb N$ be a set with $\dup(A)>0$ and $m$ a fixed positive integer. For any positive integer $k$ consider the set
	$$\textnormal{AP}(k)\eqdef\{a\in\mathbb N: a,a+k,a+2k,\ldots,a+mk \in A\}.$$
There is a positive integer $k$ such that the set $\textnormal{AP(k)}$ is infinite.
\end{theorem}

In the recent years several authors have given generalizations of Szemer\'edi's theorem by looking for patterns of the form
$$a,a+[\gamma_1(\ell)],a+[\gamma_2(\ell)],\ldots,a+[\gamma_m(\ell)]$$
in sets of positive upper density. Here $\gamma_1,\gamma_2,\ldots,\gamma_m$ are appropriate functions and $[x]$ denotes the integer part of a real number. A typical theorem in this direction is the following result from \cite{BELI}.

\begin{theorem}\label{t.polszem}Let $A\subset \mathbb N$ be a set of positive upper density, $m$ a positive integer and $p_1,p_2,\ldots,p_m$ polynomials with rational coefficients taking integer values on the integers and satisfying $p_j(0)=0$ for all $j=1,2,\ldots,m$. For $k\in\mathbb N$ consider the set	
$$	\textnormal{AP}^{pol}(k)\eqdef\{ a\in\mathbb N: a,a+p_1(k),a+p_2(k),\ldots,a+p_m(k)\in A\}.$$
There exists a positive integer $k$ with $p_j(k)\neq 0$ for all $j=1,2,\ldots,m$, such that the set $\textnormal{AP}^{pol}(k)$ is infinite.
\end{theorem}

More general versions of this theorem can be found for example in \cite{Fra}, where the functions $\gamma_1,\ldots,\gamma_m$ are allowed to be `logarithmico-exponential' functions of polynomial growth. See also \cite{FW}.

%%%%%%%%%%%%%%%%%%%%%%%%%%%%% SUBSECTION  SUBSECTION SUBSECTION
%%%%%%%%%%%%%%%%%%%%%%%%%%%%% SUBSECTION  SUBSECTION SUBSECTION
\subsection{ Weaker properties that follow from frequent universality}\label{ss.weakerproperties} Consider a sequence of operators $(T_n)_{n\in N}$ which is frequently universal. It is then immediate that $(T_n)_{n\in\mathbb N}$ satisfies the following property:

\begin{definition} Let $(T_n)_{n\in\mathbb N}$ be a sequence of operators. We will say that $(T_n)_{n\in\mathbb N}$ has property $\mathcal A$ if for every open set $U\subset X$ there exists a vector $x\in X$ such that 
	$$\dup(\{n\in\mathbb N: T_nx\in U\})>0.$$
If the sequence of operators $(T^n)_{n\in \mathbb N}$ has property $\mathcal A$ then we simply say that $T$ has property $\mathcal A$.
\end{definition}

For a single operator, property $\mathcal A$ has already been introduced in \cite[Proposition 4.6]{BADGR} where the authors show that if $T$ has property $\mathcal A$ and $T$ is hypercyclic then $T$ satisfies the hypercyclicity criterion. From \cite{BePe} this is equivalent to $T\oplus T$ being hypercyclic. Alternatively one could consider a notion of \emph{frequent recurrence} which seems to be relevant in this context:

\begin{definition} Let $(T_n)_{n\in\mathbb N}$ be a sequence of operators. A vector $x\in X$ is called a \emph{U-frequently recurrent vector for $(T_n)_{n\in\mathbb N}$} if for every open neighborhood $U_x$ of $x$, the set
	$$\{n\in\mathbb N: T_n x\in U_x\} $$
has positive upper density. The sequence $(T_n)_{n\in\mathbb N}$ is called \emph{U-frequently recurrent} if it has a dense set of U-frequently recurrent vectors. If the sequence of operators $(T^n)_{n\in \mathbb N}$ is U-frequently recurrent we will just say that $T$ is U-frequently recurrent.
\end{definition}

Observe that if a sequence of operators $(T_n)_{n\in\mathbb N}$ is U-frequently recurrent then it trivially satisfies property $\mathcal A$. It turns out that the hypothesis that $(\lambda_nT^n)_{n\in\mathbb N}$ is frequently universal in Theorem \ref{t.main1} can be replaced by the weaker hypothesis that  $(\lambda_nT^n)_{n\in\mathbb N}$ is U-frequently recurrent or by the even weaker hypothesis that $(\lambda_nT^n)_{n\in\mathbb N}$ satisfies property $\mathcal A$. We chose to state our main theorem in the introduction by using the more familiar notion of frequent universality since the definitions of frequent recurrence and that of property $\mathcal A$ are relatively new in the literature and not very well understood. However, it is relatively easy to see that there exist operators that are recurrent without being U-frequently recurrent; see the comment after Proposition \ref{p.nontopolshift}. We intend to take up these issues in a subsequent work.

\subsection{Property $\mathcal A$ implies topological multiple recurrence}\label{ss.Aimplies}

First we observe that Szemer\'edi's theorem immediately implies the following:

\begin{proposition} \label{p.simple}
If $T$ has property $\mathcal A$ then $T$ is topologically multiply recurrent. In particular, every U-frequently recurrent operator is topologically multiply recurrent.
\end{proposition}
\begin{proof}
Take a non-empty open set $U$ and fix a positive integer $m$. Since $T$ has property $\mathcal A$ there exists a vector $x\in X$ such that the set
$$A=\{ n\in\mathbb N:T^nx \in U \}$$
has positive upper density. By Szemer\'edi's theorem there exist positive integers $a$, $k$ such that
$$T^ax, T^{a+k}x, T^{a+2k}x, \ldots ,T^{a+mk}x \in U.$$ From the last we get
$$T^ax\in U\cap T^{-k}U \cap \ldots \cap T^{-km}U,$$
that is $T$ is topologically multiply recurrent.
\end{proof}

For the proof of Theorem \ref{t.main1} we need an easy technical lemma that will allow us to reduce to the case that the sequence $(\lambda_n)_{n\in\mathbb N}$ consists of positive numbers.

%%%%%%%%%%%%%%%%%%%%%%%%%%%%%% LEMMA LEMMA LEMMA
\begin{lemma}\label{l.rotations} Let $(T_n)_{n\in\mathbb N}$ be a sequence of operators acting on $X$. Then the following are equivalent:
	\begin{list}{}{}
		\item [(i)] The sequence $(T_n)_{n\in\mathbb N}$ has property $\mathcal A$.
		\item [(ii)] For every sequence of real numbers $(\theta_n)_{n\in\mathbb N}$, the sequence $(e^{i\theta_n} T_n)_{n\in\mathbb N}$ satisfies property $\mathcal A$.
	\end{list}
\end{lemma}
%%%%%%%%%%%%%%%%%%%%%%%%%%%%%% LEMMA LEMMA LEMMA

%%%%%%%%%%%%%%%%%%%%%%%%%%%%%% PROOF PROOF PROOF
\begin{proof} Obviously it is enough to prove that \textit{(i)} implies \textit{(ii)}. Let $U$ be any open set in $X$ and assume that $B(y,\epsilon)\subset U$. Since $(T_n)_{n\in\mathbb N}$ has property $\mathcal A$ there exists $x\in X$ such that the set
	\begin{align*}
		A\eqdef \{ n\in\mathbb N: T_n x\in B(y,\epsilon/2)\},
	\end{align*}
has positive upper density. Now let us write $\mathbb T= \cup_{\nu=1} ^M I_\nu$, where the arcs $I_1,\ldots,I_M$ are pairwise disjoint and each has length less than $\epsilon/(2\|y\|+\epsilon)$. We define the sets $J_\nu \eqdef \{n\in\mathbb N: e^{i\theta_n}\in I_\nu\}$. Since
\begin{align*}
	0<\dup(A)\leq \sum_{\nu=1} ^M \dup(A\cap J_\nu),
\end{align*}
we must have $\dup(A\cap J_{\nu_o})>0$ for some $\nu_o\in\{1,2,\ldots,M\}$. If $e^{i\theta_o}$ is the center of $J_{\nu_o}$ we set $z\eqdef e^{-i\theta_o}x$ and
\begin{align*}
	B\eqdef\{n\in\mathbb N: e^{i\theta_n}T_n z\in B(y,\epsilon) \}.
\end{align*}
Let $n\in A\cap J_{\nu_o}$. We have that 
\begin{align*}
	\| e^{i\theta_n}T_nz-y\|& \leq |e^{i\theta_o} - e^{i\theta_n} | \|T_n x \| + \| T_n x -y\|<\epsilon,
\end{align*}
which proves that $ A\cap J_{\nu_o} \subset B$. Thus $\dup(B)\geq \dup(A\cap J_{\nu_o})>0$ which obviously implies that the set $\{n\in\mathbb N: e^{i\theta_n} T_n z \in U\}$ has positive upper density.	
\end{proof}
%%%%%%%%%%%%%%%%%%%%%%%%%%%%%% PROOF PROOF PROOF

Theorem \ref{t.main1} is an immediate consequence of the following:

\begin{theorem} \label{t.secondary}
Let $(\lambda_n)_{n\in\mathbb N}$ be a sequence of complex numbers such that
$$\lim_{n\to \infty}\frac{|\lambda_n|} {|\lambda_{n+\tau}|}=1,$$
for some positive integer $\tau$. Assume that the sequence $(\lambda_n T^n)_{n\in\mathbb N}$ has property $\mathcal A$. Then $T$ is topologically multiply recurrent.
\end{theorem}

%%%%%%%%%%%%%%%%%%%%%%%%%%%%%% PROOF PROOF PROOF
\begin{proof} We begin by fixing a sequence $(\lambda_n)_{n\in\mathbb N}$ which satisfies the condition in the statement of the theorem as well as an operator $T$ such that the sequence $(\lambda_n T^n)_{n\in\mathbb N}$ has property $\mathcal A$. By Lemma \ref{l.rotations} this is equivalent to the sequence $(|\lambda_n| T^n)_{n\in\mathbb N}$ having property $\mathcal A$. We can and will therefore assume that $(\lambda_n)_{n\in\mathbb N}$ is a sequence of positive numbers such that
	\begin{align}\label{e.rate}
	 \lim_{n\to+\infty} \frac{\lambda_n}{\lambda_{n+\tau}}=1,
	\end{align}
and that the sequence $(\lambda_n T^n)_{n\in\mathbb N}$ has property $\mathcal A$. In order to show that $T$ is topologically multiply recurrent we fix a non-empty open set  $U$ in $X$ and a positive integer $m$ and we need to find a vector $u\in U$ and a positive integer $\ell$ with  
$$ T^\ell u, T^{2\ell}u,\ldots,T^{m\ell}u \in U.$$ 
Since $U$ is open there is a $y\in U$ and a positive number $\epsilon >0$ such that $B(y,\epsilon )\subset U$. Since $(\lambda_n T^n)_{n\in\mathbb N}$ has property $\mathcal A$ there exists some $x\in X$ such that the set 
\begin{align*}
	F\eqdef \{n\in \mathbb N: \lambda_n T^n x \in B(y,\epsilon/2)\},
\end{align*}
has positive upper density. For $k\in\mathbb N$ we define the set
\begin{align*}
	\textnormal{AP}(k)\eqdef \{a\in\mathbb N: a,a+\tau k,a+2\tau k,\ldots,a+m\tau k \in F\}.
\end{align*}
By Szemer\'edi's theorem, Theorem \ref{t.szem2}, there exists a $k\in\mathbb N$ such that the set $\textnormal{AP}(k)$ is infinite. We fix such a $k$ so that 
$$a,a+\tau k,a+2\tau k,\ldots,a+m\tau k \in F$$
for all $a\in \textnormal{AP}(k)$. Thus the vectors 
\begin{align*}
	u\eqdef \lambda_{a} T^{a}x, u_j\eqdef \lambda_{a+j\tau k}T^{a+j\tau k} =\frac{\lambda_{a+j\tau k}}{\lambda_a}T^{j\tau k}u,\quad j=1,\ldots,m,
\end{align*}
belong to $B(y,\epsilon/2)$ for all $a\in \textnormal{AP}(k)$. Now for every $j\in \{ 1,2,\ldots,m\}$ we have that
\begin{align*}
\| T^{j\tau k}u-u_j \| &=\bigg\| \frac{\lambda_{a}}{\lambda_{a+j\tau k}} u_j-u_j\bigg\|=\ABs{ \frac{\lambda_{a}}{\lambda_{a+j \tau k}} -1} \|u_j\|,
\end{align*}
and, since $\lim_{n\to+\infty} \frac{\lambda_n}{\lambda_{n+\tau}}=1$, we also have that $\lim_{a\to+\infty}\frac{\lambda_a}{\lambda_{a+j\tau k}}=1$. Thus by choosing $a$ large enough in $\textnormal{AP}(k)$ we get that
\begin{align*}
\| T^{j\tau k}u-u_j \|< \frac{\epsilon}{2},
\end{align*}
for every $j=1,\ldots,m$. Since
$$\| u_j-y\| <\frac{\epsilon}{2},$$
we conclude that
$$\| T^{j\tau k}u-y\|<\epsilon $$
for every $j=1,\ldots,m$. Therefore, with $\ell\eqdef \tau k$, we have that
$$T^{j\ell} u \in  U \quad \textrm{for every} \quad j\in\{0,1,2,\ldots ,m\},$$
as we wanted to show.
\end{proof}
%%%%%%%%%%%%%%%%%%%%%%%%%%%%%% PROOF PROOF PROOF
Some remarks are in order.

\begin{remark} When submitted, the original manuscript contained a much more restricted version of Theorem \ref{t.secondary} which only dealt with the case $\lambda_n=1/n$. The proof, albeit not so different from the one presented here, was unnecessarily complicated since it didn't use the full force of Szemer\'edi's theorem. Even with the simple form of Szemer\'edi's theorem, Theorem \ref{t.szem1}, our original proof was substantially improved by the careful reading and suggestions of the anonymous referee. The present version of the argument, used in the proof of Theorem \ref{t.secondary}, was discovered by the authors while the paper was under review and eventually replaced the original one.  
\end{remark}

\begin{remark} The notion of hypercyclicity generalizes to sequences of operators as follows. A sequence of operators $(T_n)_{n\in\mathbb N}$ is called \emph{universal} if there exists a vector $x\in X$ such that the set $\{ T_n x:n\in\mathbb N\}$ is dense in $X$. It is not hard to see that Theorem \ref{t.main1} fails if the hypothesis that $(\lambda_n T^n)_{n\in\mathbb N}$ is frequently universal is replaced by the hypothesis that $(\lambda_n T^n)_{n\in\mathbb N}$ is universal. 
	
	Indeed, let $T$ be the bilateral weighted shift acting on $l^2(\mathbb{Z})$ with weight sequence $w_n=1$ if $n\leq 0$ and $w_n=2$ if $n\geq 1$. In \cite[Example 3.6]{Leon} Le\'on-Saavedra proved that $( \frac{1}{n}T^n)_{n\in \mathbb N}$ is universal and $T$ is not hypercyclic. Assume that $T$ is recurrent. Then there exists a non-zero vector $x\in l^2(\mathbb{Z})$ which is a limit point of its orbit under $T$. From the results in \cite{ChSe} it follows that $T$ is hypercyclic, a contradiction. Thus $T$ is not recurrent. For the definition of a bilateral weighted shift on $l^2(\mathbb{Z})$ see Section \ref{s.shifts}.
\end{remark}

\begin{remark}The hypothesis that
	$$\lim_{n\to \infty}\frac{|\lambda_{n}|}{|\lambda_{n+\tau}|}=1$$
in Theorems \ref{t.secondary}, \ref{t.main1} cannot be replaced by the hypothesis
	$$\lim_{\stackrel{n\to \infty}{n\in A}}\frac{|\lambda_{n}|}{|\lambda_{n+\tau}|}=1,$$
for some set $A\subset \mathbb N$ which has positive density $0<\dlo(A)=\dup(A)\leq 1$, as the following examples show.
\end{remark}

\begin{example}
	Consider the set $A=\{ 2n: n=1,2,\ldots \}$ and define the sequence $ (\lambda_n )_{n\in\mathbb N}$ by $\lambda_{2n}=2^n$ for $n=1,2,\ldots $ and $\lambda_{2n+1}=2^n$ for $n=0,1,2,\ldots $. Then we have $\lim_{n\in A, n\to +\infty}\frac{\lambda_n}{\lambda_{n+1}}=1$ and $\dlo(A)=\dup(A)=1/2$. We shall prove that the sequence $(\lambda_nB^n)_{n\in\mathbb N}$ is frequently universal, where $B$ is the unweighted unilateral backward shift acting on $l^2(\mathbb{N})$. Indeed, since $\lambda_{2n}B^{2n}=(2B^2)^n$ it follows, by \cite{BaGr2}, that $2B^2$ is frequently hypercyclic and therefore the sequence $( \lambda_{2n}B^{2n} )_{n\in\mathbb N}$ is frequently universal. An easy argument now shows that the sequence $ (\lambda_nB^n )_{n\in\mathbb N}$ is frequently universal. Hence $ (\lambda_nB^n )_{n\in\mathbb N}$ has property $\mathcal A$ and on the other hand $B$ is not recurrent.  
\end{example}

\begin{example}
We now present a stronger example than the previous one, in the following sense. There exist a sequence $(\lambda_n )_{n\in\mathbb N}$ of positive integers, a subset $A$ of $\mathbb{N}$ with $\dlo(A)=\dup(A)=1$ and an operator $T$ acting on $l^2(\mathbb{N})$ such that $\lim_{n\in A, n\to +\infty}\frac{\lambda_n}{\lambda_{n+1}}=1$, the sequence $(\lambda_nT^n )$ is frequently universal but $T$ is not topologically multiply recurrent. Define $\lambda_n=2^{2^k}$ if $n\in [2^{k-1}, 2^k)$. It can be easily checked that the set $A\eqdef \cup_{k=2} ^{+\infty} ([2^{k-1}, 2^k-2]\cap \mathbb{N})$ has density $1$ and $\lim_{n\in A, n\to +\infty}\frac{\lambda_n}{\lambda_{n+1}}=1$. Using the frequent universality criterion from \cite{BG} and \cite{BonGroE} it is not difficult to show that $(\lambda_nB^n)_{n\in\mathbb N}$ is frequently universal, where $B$ is the unweighted unilateral backward shift. Clearly $B$ is not recurrent.       
\end{example}

%%%%%%%%%%%%%%%%%%%%%%%%%%%%% SUBSECTION  SUBSECTION SUBSECTION
%%%%%%%%%%%%%%%%%%%%%%%%%%%%% SUBSECTION  SUBSECTION SUBSECTION
\subsection{Polynomial multiple recurrence}\label{s.multrecur}
 It is immediate from the proof of Theorem \ref{t.secondary} and from the discussion on the different forms of Szemer\'edi's theorem that a stronger version of Theorem \ref{t.secondary} should hold true. This is indeed the case:

\begin{theorem} Let $(\lambda_n)_{n\in\mathbb N}$ be a sequence of non-zero complex numbers which satisfies
\begin{align*}
	\lim_{n\to+\infty}\frac{|\lambda_n|}{|\lambda_{n+\tau}|}=1,
\end{align*}
for some positive integer $\tau$. Let $T$ be an operator acting on $X$ such that the sequence $(\lambda_n T^n)_{n\in\mathbb N}$ has property $\mathcal A$. Let $U$ be an open set in $X$ and $m$ a positive integer. Let $p_1,p_2,\ldots,p_m$ be polynomials with rational coefficients taking positive integer values on the positive integers and satisfying $p_j(0)=0$ for all $j=1,2,\ldots,m$. Then there is a positive integer $k$ such that
	$$ U \cap T^{-p_1(k)}(U)\cap\cdots\cap T^{-p_m(k)}(U)\neq \emptyset.$$
\end{theorem}
The proof of this theorem is omitted as it is just a repetition of the proof of Theorem \ref{t.secondary} where one uses the polynomial Szemer\'edi theorem, Theorem \ref{t.polszem}, instead of Theorem \ref{t.szem2}.

%%%%%%%%%%%%%%%%%%%%%%%%%%%%% SUBSECTION  SUBSECTION SUBSECTION
%%%%%%%%%%%%%%%%%%%%%%%%%%%%% SUBSECTION  SUBSECTION SUBSECTION
%%%%%%%%%%%%%%%%%%%%%%%%%%%%%% SECTION  SECTION SECTION
%%%%%%%%%%%%%%%%%%%%%%%%%%%%%% SECTION  SECTION SECTION 
\section{Good and bad sequences}\label{s.bad}
In this section we give examples of sequences for which our main theorem is valid as well as examples that exhibit that if the limit
\begin{equation}\label{e.limit}
	\lim_{n\to+\infty}|\lambda_n|/|\lambda_{n+\tau}|
\end{equation}
exists but is different than $1$, then Theorem \ref{t.main1} fails in general. 

We shall say that a sequence $(\lambda_n)_{n\in\mathbb N} $ of complex numbers is \emph{good} if Theorem \ref{t.secondary} holds true for this sequence; that is, the sequence $(\lambda_n)_{n\in\mathbb N}$ is \emph{good} if, for any bounded linear operator $T$ acting on $X$, the operator $T$ is topologically multiply recurrent whenever the sequence $ (\lambda_n T^n )_{n\in\mathbb N}$ has property $\mathcal A$. Otherwise we will say that $(\lambda_n)_{n\in\mathbb N}$ is \emph{bad}. 

\subsection{Bad sequences}\label{ss.bad} Here we give the proof of Proposition \ref{p.bad}.

%%%%%%%%%%%%%%%%%%%%%%%%%%%%%% PROOF PROOF PROOF
\begin{proof}[Proof of Proposition \ref{p.bad}] A moment's reflection shows that it is enough to consider the case $\tau=1$. We first assume that $a$ is a complex number with $0<|a|<1$ and let $B:l^2(\mathbb N)\to l^2(\mathbb N)$ be the unweighted backward shift, that is 
	$$ B(w_1,w_2,\ldots)=(w_2,w_3,\ldots),$$	
for all sequences $(w_1,w_2,\ldots)\in l^2(\mathbb{N})$. We define $w\eqdef a^{-\frac{1}{2}}$ and consider the sequence $\lambda_n=w^{2n}$. We have that $\lambda_{n}/\lambda_{n+1}=w^{-2}=a$. Observe that 
$$|w|^2=|a|^{-1}>1\implies |w|>1$$ 
so we can define the operator $T\eqdef \frac{1}{w}B$ and we have $\|T\|<1$. On the other hand, since $|w|>1$ a result from \cite{BaGr2} shows that $wB$ is frequently hypercyclic which is equivalent to saying that the sequence $(w^n B^n)_{n\in \mathbb N}$ is frequently universal. Observe that we have
$$w^n B^n=w^{2n}\frac{1}{w^n}B^n =\lambda_n T^n .$$
That is, the sequence $(\lambda_nT^n)_{n\in \mathbb N}$ is frequently universal. However $T$ is not recurrent since $\|T\|<1$. 

Let now $a\in\mathbb C$ with $|a|>1$. Let $S=\{re^{i\theta}\in\mathbb C:1<r<2|a|,0<\theta<\pi/2\}$ and $\phi:\mathbb D \to S$ be the Riemann map of $\mathbb D$ onto $S$. Consider the sequence $\lambda_n= a ^{-n}$. Let $\mathbb H^2(\mathbb D)$ be the Hardy space on the unit disc and $M_\phi$ be the multiplication operator on $\mathbb H^2(\mathbb D)$, for the function $\phi$ just defined. It is well known that since $\phi(\mathbb D)\cap  \mathbb T=\emptyset$, the adjoint of the multiplication operator, $M^* _\phi$ is \emph{not} hypercyclic \cite{GoSh}. Moreover, by Proposition \ref{p.adjequiv} we know that $M^* _\phi$ is not even recurrent in this case. On the other hand it is easy to see that
$$\lambda_n (M^* _\phi)^n=(a^{-1} M_\phi ^*)^n = (M^* _{\phi/\bar a})^n=(M^* _\psi)^n,$$
where $\psi=\phi/\bar a$. Since $\psi(\mathbb D)\cap  \mathbb T\neq \emptyset$ we conclude that $M^* _\psi$ is frequently hypercyclic. For this see \cite{BaGr2} or Proposition \ref{p.adjequiv} of the present paper. However, this means that the sequence $\lambda_n (M^* _\phi)^n$ is frequently universal. Since $M^* _\phi$ is not recurrent, thus not topologically multiply recurrent, this completes the proof in this case as well. Actually the same argument works also for the case $|a|<1$. For this just consider the Riemann map of $\mathbb D$ onto $S'=\{re^{i\theta}\in\mathbb C: |a|/2<r<1, 0<\theta<\pi/2\}$.

Let us now move to the case $a=0$. We set $\lambda_n=n!$ so $\lambda_n/\lambda_{n+1}=1/(n+1)\to 0 $ as $n\to +\infty$. Now let $B$ be the backward shift on $l^2(\mathbb N)$ as before. It is not hard to see that the sequence $\lambda_n B^n$ satisfies the frequent universality criterion from \cite{BG} and \cite{BonGroE}. We conclude that $\lambda_n B^n$ is frequently universal. However $B$ is not recurrent since $\|B^nx\|\to 0$ as $n\to+\infty$ for all $x\in X$.
\end{proof}
%%%%%%%%%%%%%%%%%%%%%%%%%%%%%% PROOF PROOF PROOF

\begin{remark} It is not hard to see that if $|\lambda_n|/|\lambda_{n+1}|\to +\infty$ and $T$ is any bounded linear operator then the sequence $(\lambda_nT^n)_{n\in\mathbb N}$ is never frequently universal. To see this let $T$ be any operator and fix some positive integer $n_o$ such that $|\lambda_n|/|\lambda_{n+1}|> 1+\|T\|$ for all $n\geq n_o$. Observe that for all $n\geq n_o$ we have
	\begin{align*}
	\bigg|\frac{\lambda_{n}}{\lambda_{n_o}}\bigg|=\bigg|\frac{\lambda_{n}}{\lambda_{n-1}}\cdots \frac{\lambda_{n_o+1}}{\lambda_{n_o}}\bigg|\leq (1+\|T\|)^{-(n-n_o)}.
	\end{align*}
Now for any $x\in X$ we conclude that
\begin{align*}
	\| \lambda_n T^n x \| \leq |\lambda_{n_o}|(1+\|T\|)^{n_o} \bigg(\frac{\|T\|}{1+\|T\|}\bigg)^n,
\end{align*}	
for all $n\geq n_o$. Letting $n\to +\infty$ we get that $\lim_{n\to+\infty}	\| \lambda_n T^n x \|= 0$ for all $x\in X$. Thus $(\lambda_nT^n)_{n\in \mathbb N}$ cannot even be recurrent in this case.
\end{remark}

\subsection{Good sequences}In this subsection we shall present several illustrating examples for which Theorem \ref{t.secondary} applies.

\begin{example}[Sub-Polynomial Growth]
Theorem \ref{t.secondary} implies that the following sequences, $(\log n)_{n\in\mathbb N}$, $( \log \log n)_{n\in\mathbb N}$, $( (\log n)^k)_{n\in\mathbb N}$, $k\in \mathbb{R}$, are good sequences.   
\end{example}

\begin{example} [Polynomial Growth]
If $Q$ is a non-zero rational complex-valued function then the sequence $(Q(n))_{n\in\mathbb N}$ is good. In particular if $P$ is any non-zero polynomial the sequence $ (P(n))_{n\in\mathbb N}$ is good. This is straightforward from Theorem \ref{t.secondary} since $Q(n)/Q(n+1)\to 1$ as $n\to +\infty$. 
\end{example}

\begin{example}[Super-polynomial Growth]\label{suppol}
For a real number $a$ the sequence $(e^{n^a})_{n\in\mathbb N}$ is good if and only if $a<1$. Indeed, if $a<1$ then $e^{n^a}/e^{(n+1)^a}\to 1$ as $n\to +\infty$ and by Theorem \ref{t.secondary} we conclude that the sequence $ (e^{n^a})_{n\in\mathbb N} $ is good. Observe that for $0<a<1$ the sequence $\{ e^{n^a} \}$ has super-polynomial growth. On the other hand for $a=1$, the sequence $ (e^{n}B^n)_{n\in\mathbb N}$ is frequently hypercyclic see \cite{BaGr2}, where $B$ is the unweighted unilateral backward shift acting on the space of square summable sequences, and clearly $B$ is not even recurrent. Therefore $(e^n)_{n\in\mathbb N}$ is a bad sequence and from the discussion in paragraph \ref{ss.bad} it follows that for every $a>1$ the sequence $(e^{n^a})_{n\in\mathbb N}$ is bad as well.     
\end{example}

\begin{example}
Let us see an example of a good sequence $( \lambda_n )_{n\in\mathbb N}$ which grows faster than every sequence $(e^{n^a})_{n\in\mathbb N}$ , $0<a<1$. Indeed, just take $\lambda_n=e^{\frac{n}{\log n}}$. Since $\lambda_n/\lambda_{n+1}\to 1$ it follows that $(e^{\frac{n}{\log n}})_{n\in\mathbb N}$ is a good sequence. Observe that the sequences $( e^{\frac{n}{\log \log n}})_{n\in\mathbb N}$,  $(e^{\frac{n}{\log \log \log n}})_{n\in\mathbb N}$, etc. are good sequences as well.  
\end{example}

%%%%%%%%%%%%%%%%%%%%%%%%%%%%% SUBSECTION  SUBSECTION SUBSECTION
%%%%%%%%%%%%%%%%%%%%%%%%%%%%% SUBSECTION  SUBSECTION SUBSECTION

%%%%%%%%%%%%%%%%%%%%%%%%%%%%% SECTION SECTION SECTION
%%%%%%%%%%%%%%%%%%%%%%%%%%%%% SECTION SECTION SECTION
\section{Weighted shifts} \label{s.shifts}
In this paragraph we give a characterization of topologically multiply recurrent bilateral weighted shifts in terms of their weight sequences.

Let $l^2(\mathbb{N})$ be the Hilbert space of square summable sequences $x=(x_n)_{n\in \mathbb{N}}$. Consider
the canonical basis $(e_n)_{n\in \mathbb{N}}$ of $l^2(\mathbb{N})$ and let $(w_n)_{n\in \mathbb{N}}$ be a
(bounded) sequence of positive numbers. The operator $T:l^2(\mathbb{N})\to l^2(\mathbb{N})$ is a
\emph{unilateral (backward) weighted shift} with weight sequence $(w_n)_{n\in\mathbb N}$ if $Te_{n}=w_ne_{n-1}$ for every $n\geq
1$ and $Te_1=0$.

Let $l^2(\mathbb{Z})$ be the Hilbert space of square summable sequences $x=(x_n)_{n\in \mathbb{Z}}$ endowed with
the usual $l^2$ norm. That is, $x=(x_n)_{n\in \mathbb{Z}}\in l^2(\mathbb{Z})$ if
$\sum_{n=-\infty}^{+\infty}|x_n|^2<+\infty $. Let $(w_n)_{n\in \mathbb{Z}}$ be a (bounded) sequence of positive
numbers. The operator $T:l^2(\mathbb{Z})\to l^2(\mathbb{Z})$ is a \emph{bilateral (backward) weighted shift}
with weight sequence $(w_n)_{n\in\mathbb Z}$ if $Te_{n}=w_ne_{n-1}$ for every $n\in \mathbb{Z}$. Here $(e_n)_{n\in \mathbb{Z}}$
is the canonical basis of $l^2(\mathbb{Z})$.

We begin by showing that for bilateral weighted shifts, hypercyclicity is equivalent to recurrence.
\begin{proposition}
Let $T:l^2(\mathbb{Z})\to l^2(\mathbb{Z})$ be a bilateral weighted shift. Then $T$ is hypercyclic if and only if
$T$ is recurrent.
\end{proposition}
\begin{proof}
Let $(w_n)_{n\in\mathbb Z}$ be the weight sequence of $T$. We only have to prove that if $T$ is recurrent then $T$ is hypercyclic
since the converse implication holds trivially. So assume $T$ is recurrent. Let $q$ be a positive integer and
consider $\epsilon >0$. Choose $\delta
>0$ such that $\delta /(1-\delta)<\epsilon $ and $\delta <1$. Consider the open ball $B(\sum_{|j|\leq q}e_j,\delta )$.
There exists a positive integer $n>2q$ such that
\[ B\bigg(\sum_{|j|\leq q}e_j,\delta \bigg)\bigcap T^{-n}\bigg( B\bigg(\sum_{|j|\leq q}e_j,\delta \bigg) \bigg)\neq \emptyset .\]
Hence there exists $x\in l^2(\mathbb{Z})$ such that
$$\bigg\| x-\sum_{|j|\leq q}e_j\bigg\| <\delta$$ and
$$ \bigg\| T^nx-\sum_{|j|\leq q}e_j\bigg\| <\delta .$$
Having at our hands the last inequalities we argue as in the proof of Theorem 2.1 in
\cite{Salas} and we conclude that for all $|j|\leq q$
$$ \prod_{s=1}^{n}w_{s+j}>\frac{1}{\epsilon}\quad \textrm{and} \quad \prod_{s=0}^{n-1}w_{j-s}< \epsilon .$$
The last conditions are known to be sufficient for $T$ to be hypercyclic; see \cite{Salas}.
\end{proof}

In the case of weighted bilateral shifts we provide a characterization of topological multiple recurrence in
terms of the weights. Another characterization can be given through the notion of $d$-hypercyclic operators
introduced by Bes and Peris in \cite{BePe2} and independently by Bernal-Gonzalez in \cite{Bernal}:
\begin{definition}
Let $N\geq 2$ be a positive integer. The operators $T_1,\ldots ,T_N$ acting on $X$ are called \emph{disjoint} or
\emph{diagonally hypercyclic} (in short \emph{$d$-hypercyclic}) if there exists a vector $x\in X$ such that
\[ \overline{ \{ (T_1^nx, \ldots ,T_N^nx): n=0,1,2,\ldots \} }=X^N.\] If the set of such vectors $x$ is dense in
$X$ then the operators $T_1,\ldots ,T_N$ are called \emph{densely $d$-hypercyclic}.
\end{definition}
\begin{proposition} \label{p.equiv}
Let $T:l^2(\mathbb{Z})\to l^2(\mathbb{Z})$ be a bilateral weighted shift with weight sequence $(w_n)_{n\in\mathbb Z}$. The
following are equivalent:
\begin{list}{}{}
\item [(i)] $T$ is topologically multiply recurrent.

\item [(ii)] For every $m\in \mathbb{N}$ the operator $T\oplus T^2\oplus \cdots \oplus T^m$ is hypercyclic on
$X^m$.

\item[(iii)] For every $m\in \mathbb{N}$ the operators $T,T^2,\ldots , T^m$ are densely $d$-hypercyclic.

\item[(iv)] For every $m,q\in \mathbb{N}$ and every $\epsilon >0$ there exists a positive integer
$n=n(m,q,\epsilon )$ such that for every integer $j$ with $|j|\leq q$ and every $l=1,\ldots ,m$ we have
\[ \prod_{i=1}^{ln}w_{j+i}>\frac{1}{\epsilon } \quad \mbox{and} \quad
\prod_{i=0}^{ln-1}w_{j-i}<\epsilon . \]
\end{list}
If in addition $T$ is invertible then any of the conditions \textit{(i)-(iv)} is equivalent to:
\begin{list} {}{}
\item [(v)] for every $m\in \mathbb{N}$ there exists a strictly increasing sequence of positive integers $\{
n_k \}$ such that
\[ \lim_{k\to +\infty}\prod_{i=1}^{ln_k}w_{i}=\lim_{k\to
+\infty}\prod_{i=0}^{ln_k}\frac{1}{w_{-i}}=+\infty, \] for every $l=1,\ldots ,m$.
\end{list}

\end{proposition}
\begin{proof}
Let us first prove that \textit{(i)} implies \textit{(iv)}. Fix a positive integer $m$ and let $\epsilon >0$. Take also a
positive integer $q$. Then consider a positive number $\delta $ such that $\delta /(1-\delta)<\epsilon $ and
$\delta <1$. Consider the open ball $B(\sum_{|j|\leq q}e_j,\delta )$. There exists a positive integer $n>2q$
such that
\[ B\bigg(\sum_{|j|\leq q}e_j,\delta \bigg)\bigcap
T^{-n}\bigg( B\bigg(\sum_{|j|\leq q}e_j,\delta \bigg) \bigg) \bigcap \cdots \bigcap T^{-mn}\bigg(
B\bigg(\sum_{|j|\leq q}e_j,\delta \bigg) \bigg) \neq \emptyset .\] 
Hence there exists $x=(x_k)_{k\in \mathbb{Z}}
\in l^2(\mathbb{Z})$ such that
$$\bigg\| T^{ln}x-\sum_{|j|\leq q}e_j\bigg\| <\delta$$
for every $l=0,1,\ldots ,m$. Testing the previous condition for $l=0$, and since $n>2q$, we conclude that
necessarily
$$|x_k|<\delta \quad \textrm{for} \,\, |k|\geq n .$$
Having at our hands the above inequalities we argue as in the proof of Theorem 4.7 in \cite{BePe2} and we
conclude that for all $|j|\leq q$ and for all $l=1,\ldots ,m$
$$ \prod_{s=1}^{ln}w_{s+j}>\frac{1}{\epsilon} \quad \textrm{and} \quad \prod_{s=0}^{ln-1}w_{j-s}<\epsilon  .$$
Hence we proved that \textit{(i)} implies \textit{(iv)}. Condition \textit{(ii)}, \textit{(iii)} and \textit{(iv)} are known to be equivalent from
Proposition 4.8 and Corollary 4.9 in \cite{BePe2}. Finally the implication $(iii)\Rightarrow (i)$ holds
trivially and this completes the proof of the equivalence of statements \textit{(i)-(iv)} of the proposition. It remains to prove that \textit{(iv)} is equivalent to \textit{(v)} in the case that $T$ is invertible. For $l=1$ this has been done in \cite{Fel}. The case of general $l$ follows by an obvious modification of the argument in \cite{Fel}.
\end{proof}
In the context of unilateral weighted shifts, it turns out that some general classes of operators are always topologically multiply recurrent. To see this we first need to recall the following well known notions from topological dynamics.
\begin{definition}
An operator $T$ acting on $X$ is called \emph{topologically mixing} if for every pair $(U,V)$ of non-empty open
sets in $X$ there exists a positive integer $N$ such that $T^n(U)\cap V\neq \emptyset $ for every $n\geq N$.
\end{definition}

\begin{definition}
An operator $T$ acting on $X$ is called \emph{chaotic} if it is hypercyclic and its set of periodic points, i.e
the set $\{ x\in X: T^nx=x \quad \mbox{for some} \quad n\in \mathbb{N} \}$, is dense in $X$.
\end{definition}

\begin{corollary} \label{c.mix}
Let $T:l^2(\mathbb{N})\to l^2(\mathbb{N})$ be a unilateral weighted shift with weight sequence $(w_n)_{n\in\mathbb N}$. If
$T$ is topologically mixing or chaotic then $T$ is topologically multiply recurrent.
\end{corollary}
\begin{proof}
Suppose first that $T$ is topologically mixing. Take a positive integer $m\geq 2$. Then it is easy to show that
$T\oplus T^2\oplus \cdots \oplus T^m$ is hypercyclic and by the analogue of Proposition \ref{p.equiv} for
unilateral shifts the conclusion follows. Assume now that $T$ is chaotic. Then the weight sequence $(w_n)_{n\in\mathbb N}$
satisfies the condition
\[ \sum_{n=1}^{+\infty}(w_1\cdots w_n)^{-2} <+\infty .\]
See for instance Theorem 6.12 in \cite{BM}. The last condition is known to be sufficient for $T$ to be
frequently hypercyclic; see \cite{BaGr2}. By Proposition \ref{p.simple} we conclude that $T$ is topologically
multiply recurrent.
\end{proof}

\begin{remark}
In $c_0(\mathbb{N})$ there exists a unilateral weighted shift which is
frequently hypercyclic and thus topologically multiply recurrent but is neither chaotic nor mixing. Such an
example is provided in \cite{BaGr3}.
\end{remark}

\begin{proposition}\label{p.nontopolshift}
There exists a hypercyclic bilateral weighted shift on $l^2(\mathbb Z)$ which is not topologically multiply recurrent.
\end{proposition}
\begin{proof}
Take a hypercyclic bilateral weighted shift $T$ acting on $l^2(\mathbb{Z})$ such that $T\oplus T^2$ is not
hypercyclic. Examples of such operators are provided, for instance, in Theorem 1.3 of \cite{GrRo}. By
Proposition \ref{p.equiv} the operator $T$ is not topologically multiply recurrent.
\end{proof}

\begin{remark} The operator of Proposition \ref{p.nontopolshift} provides us with an example of an operator which is recurrent but not U-frequently recurrent in view of Theorem \ref{t.main1}.
\end{remark}

We finish this section by showing that in general the converse of Proposition \ref{p.simple} is not true.
\begin{proposition}
There exists a unilateral weighted shift on $l^2(\mathbb{N})$ which is topologically multiply recurrent but not
frequently hypercyclic.
\end{proposition}
\begin{proof}
Consider the unilateral weighted shift $T$ on $l^2(\mathbb{N})$ with weight sequence $w_n=\sqrt{\frac{n+1}{n}}$,
$n=1,2,\ldots $. Then $\prod_{i=1}^nw_i=\sqrt{n+1}\to +\infty $. By the main result of \cite{CoSa} it follows
that $T$ is topologically mixing and by Corollary \ref{c.mix} the operator $T$ is topologically multiply recurrent. On
the other hand, as it is shown in Example 2.9 of \cite{BaGr2}, $T$ is not frequently hypercyclic.
\end{proof}

\section{Adjoints of multiplication operators} \label{s.adjoints}
In this section we will study adjoints of multiplication operators on suitable Hilbert spaces. As we shall see
in this case it is easy to characterize topological multiple recurrence in terms of several different well
understood conditions. Following \cite{GoSh} we fix a non-empty open connected set $\Omega$ of $\mathbb{C}^n$,
$n\in \mathbb{N}$, and $H$ a Hilbert space of holomorphic functions on $\Omega $ such that:
\begin{list}{}{}
\item[-] $H\neq \{ 0\}$, and
\item [-]for every $z\in \Omega$, the point evaluation functionals $f\to f(z)$, $f\in H$, are bounded.
\end{list}
Recall that every complex valued function $\phi:\Omega \to \mathbb{C}$ such that the pointwise product $\phi f$
belongs to $H$ for every $f\in H$ is called \emph{a multiplier of $H$}. In particular $\phi$ defines the multiplication
operator $M_{\phi}:H\to H$ in terms of the formula $$M_{\phi }(f)=\phi f,\quad f\in H.$$ By the boundedness of point
evaluations along with the closed graph theorem it follows that $M_{\phi}$ is a bounded linear operator on $H$.
It turns out that under our assumptions on $H$, every multiplier $\phi $ is a bounded holomorphic function, that
is $\| \phi \|_{\infty}:=\sup_{z\in \Omega }|\phi (z)|<+\infty $. In particular we have that $\| \phi \|_{\infty }\leq \| M_{\phi}\| $; see \cite{GoSh}.

In Proposition \ref{p.adjequiv} we require the more stringent condition $\| M_{\phi}\| =\| \phi \|_{\infty }$
as well as the condition that every non-constant bounded holomorphic function $\phi $ on $\Omega $ is a
multiplier of $H$. This is quite natural since it is actually the case in typical examples of Hilbert spaces of
holomorphic functions such as the Hardy space $H^2(\mathbb D)$ or the Bergman space $A^2(\mathbb D)$, on the unit disk $\mathbb D$. On the other
hand Proposition \ref{p.adjequiv} fails if we remove this hypothesis as can be seen by studying adjoints of
multiplication operators on Dirichlet spaces. See Example 2.4 of \cite{ChSe2}.
\begin{proposition} \label{p.adjequiv}
Suppose that every non-constant bounded holomorphic function $\phi $ on $\Omega $ is a multiplier of $H$ such
that $\| M_{\phi}\| =\| \phi \|_{\infty }$. Then for each such $\phi$ the following are equivalent.
\begin{list}{}{}
\item [(i)] $M_{\phi}^*$ is topologically multiply recurrent.

\item [(ii)] $M_{\phi }^*$ is recurrent.

\item[(iii)] $M_{\phi }^*$ is frequently hypercyclic.

\item[(iv)] $M_{\phi }^*$ is hypercyclic.

\item[(v)] $\phi (\Omega )\cap \mathbb{T}\neq \emptyset $.

\item[(vi)] $M_{\phi }^*$ has property $\mathcal A$.
\end{list}
\end{proposition}
\begin{proof}
Conditions \textit{(iii)}, \textit{(iv)} and \textit{(v)} are known to be equivalent; see \cite{BaGr2}, \cite{GoSh}. Trivially
\textit{(iii)} implies \textit{(vi)}; by Proposition \ref{p.simple} \textit{(vi)} implies \textit{(i)} and trivially \textit{(i)} implies \textit{(ii)}.  We will show that \textit{(ii)} implies \textit{(v)}.
Indeed, assume that $M_{\phi }^*$ is recurrent.  Suppose, for the sake of contradiction, that $\phi (\Omega
)\cap \mathbb{T}= \emptyset $. Since $\Omega$ is connected, so is $\phi (\Omega )$; thus, we either have that
$\phi (\Omega )\subset \{ z\in \mathbb{C}: |z|<1 \}$ or $\phi (\Omega )\subset \{ z\in \mathbb{C}: |z|>1 \}$.

\medskip
\noindent\fbox{ \textbf{Case 1.  $\phi (\Omega )\subset \{ z\in \mathbb{C}: |z|<1 \}$.}}
\medskip

 Then we have $\| M_{\phi }^*
\|=\| M_{\phi } \| =\| \phi \|_{\infty} \leq 1$. We will consider two complementary cases. Assume that there
exist $0<\epsilon<1$ and a non-zero recurrent vector $g$ for $M_{\phi }^*$ such that
\[ \| M_{\phi }^*g\| \leq (1-\epsilon)\| g\| .\]
The above inequality and the fact that $\| M_{\phi }^* \|\leq 1$ imply that for every positive integer $n$
\[ \| (M_{\phi }^*)^ng\| \leq (1-\epsilon)\| g\| .\]
On the other hand for some strictly increasing sequence of positive integers $\{ n_k\}$ we have $(M_{\phi
}^*)^{n_k}g\to g$. Using the last inequality we arrive at $\| g\| \leq (1-\epsilon )\| g\|$, a contradiction. In
the complementary case we must have $\| M_{\phi }^*g \| \geq \| g\| $ for every vector $g$ which is recurrent  for
$M_{\phi }^*$. Since the set of recurrent vectors for $M_{\phi }^*$ is dense in $H$ we get that $\| M_{\phi }^*h
\| \geq \| h\| $ for every $h\in H$. Hence $\| M_{\phi }^*h \| = \| h\| $ for every $h\in H$. Take now $z\in
\Omega $ and consider the reproducing kernel $k_z$ of $H$. Then we have that
\[ \| M_{\phi }^*k_z \| =|\phi (z)|\| k_z\|< \| k_z\| .\]
For the previous identity see Proposition 4.4 of \cite{GoSh}. However, this is clearly impossible since $M_{\phi }^*$ is an isometry.

\medskip
\noindent\fbox{ \textbf{Case 2.  $\phi (\Omega )\subset \{ z\in \mathbb{C}: |z|>1 \}$.}}
\medskip

Here $1/\phi$ is a bounded holomorphic function satisfying $\| 1/\phi \|_{\infty }\leq 1$; therefore, $M_{\phi}^*$ is invertible. By Remark
\ref{r.rec} the operator $M_{1/\phi}^{*}=(M_{\phi}^*)^{-1}$ is recurrent and the proof follows by Case 1.
\end{proof}

\begin{remark}
It is easy to see that under the hypotheses of Proposition \ref{p.adjequiv}, $M_{\phi }$ is never recurrent. On
the other hand, suppose that $\phi$ is a constant function with $\phi(z)=a$ for some $a\in\mathbb C$ and every $z\in \Omega$. Then we have that
$M_{\phi }$ (or equivalently $M_{\phi }^*$) is recurrent if and only if $M_{\phi }$ is topologically multiply
recurrent if and only if $|a|=1$. In order to prove this it is enough to notice that for every non-zero complex number $a$, with $|a|=1$, and every positive integer $m$, there exists an increasing sequence of positive integers $\{n_k\}$ such that $(a^{n_k},a^{2n_k},\ldots,a^{mn_k})\to (1,1,\ldots,1).$
\end{remark}

\section{Further Questions} \label{s.remarks}
We conclude this note by suggesting a series of questions that relate to the results and the notions discussed in the preceding paragraphs:

\begin{question}
Let $T:l^2(\mathbb{N})\to l^2(\mathbb{N})$ be a unilateral weighted shift. It is a classical result of Salas
that $I+T$ is hypercyclic; see \cite{Salas}. In fact, as observed by Grivaux in \cite{Grivaux}, $I+T$ is even
mixing. Hence it is natural to ask the following question: is it true that $I+T$ is topologically multiply
recurrent?
\end{question}

\begin{question}
Let $T$ be frequently hypercyclic. Is it true that for every
positive integer $N\geq 2$ the operator $T\oplus T^2\oplus \cdots \oplus T^m$ is hypercyclic? Recall that
Grosse-Erdmann and Peris have shown that in this case $T\oplus T$ is hypercyclic; see \cite{GP}. The fact that
$T\oplus T$ is hypercyclic is known to be equivalent to $T$ satisfying the \emph{hypercyclicity criterion}; see
\cite{BePe}. However, if $T$ is hypercyclic it is not true in general that $T\oplus T$ is hypercyclic. This was
a long standing question that was solved in \cite{DeRe} and in a more general context in \cite{BM1}. In
\cite{CoRu} it is proved that if $T$ is frequently hypercyclic then
$T^j\oplus T^m$ is hypercyclic, for every pair of positive integers $(j,m)$. We propose the following stronger
question. Let $T$ be a frequently hypercyclic operator. Is it true that the operators $T,T^2, \ldots ,T^N$ are $d$-hypercyclic for every positive integer $N\geq 2$? Observe that by
Propositions \ref{p.simple}, \ref{p.equiv} and Theorem \ref{t.secondary} this is indeed the case for bilateral weighted
shifts.
\end{question}

\begin{question}
Observe that Proposition \ref{p.adjequiv} misses the case of Hilbert spaces $H$ where not all bounded
holomorphic functions are multipliers. An example of such a space is the \emph{Dirichlet space}
$\textnormal{Dir}(\mathbb{D})$, that is the Hilbert space of holomorphic functions $f:\mathbb{D}\to\mathbb C$, satisfying
\[ \| f\|_{\textnormal{Dir}} ^2:=|f(0)|^2+\frac{1}{\pi } \int|f'(z)|^2dA(z)<+\infty ,\] where $dA$ denotes the area measure. It would be
interesting to characterize when the adjoints of multiplication operators on $\textnormal{Dir}(\mathbb{D})$ are hypercyclic,
frequently hypercyclic, recurrent, or topologically multiply recurrent. For results along this direction we refer to \cite{BSH}.
\end{question}

\begin{question} It is easy to see that every chaotic operator $T$ has property $\mathcal A$. A well known question asks whether every chaotic operator is frequently hypercyclic; see for example \cite{BM}. An even stronger question is thus whether every hypercyclic operator that has property $\mathcal A$ is frequently hypercyclic.
\end{question}

\begin{question} Let $( \lambda_n)_{n\in\mathbb N}$ be a good sequence. Is it true that there is a positive integer $\tau$  such that the limit $\lim_{n\to\infty} |\lambda_{n}|/|\lambda_{n+\tau|}$ exists? A positive answer would provide a complete characterization of good sequences.
\end{question}

\begin{question} As observed in section \ref{ss.weakerproperties} every U-frequently recurrent operator has property $\mathcal A$. Does there exist an operator which has property $\mathcal A$ but is \emph{not} U-frequently recurrent? We suspect that the answer is positive.
	
\end{question}


\begin{bibsection}
\begin{biblist}
	
	
	\bib{BADGR}{article}{
		Author = {Badea, Catalin},
		Author = {Grivaux, Sophie},
		Coden = {ADMTA4},
		Date-Added = {2010-12-03 19:29:57 +0200},
		Date-Modified = {2010-12-03 19:33:33 +0200},
		Doi = {10.1016/j.aim.2006.09.010},
		Fjournal = {Advances in Mathematics},
		Issn = {0001-8708},
		Journal = {Adv. Math.},
		Mrclass = {47A10 (11K06 37B99 47A16 47D06)},
		Mrnumber = {2323544 (2008g:47011)},
		Mrreviewer = {Teresa Berm{\'u}dez},
		Number = {2},
		Pages = {766--793},
		Title = {Unimodular eigenvalues, uniformly distributed sequences and linear dynamics},
		Url = {http://dx.doi.org/10.1016/j.aim.2006.09.010},
		Volume = {211},
		Year = {2007},
		Bdsk-Url-1 = {http://dx.doi.org/10.1016/j.aim.2006.09.010} }


	

\bib{BaGr1}{article}{
	Author = {Bayart, Fr{\'e}d{\'e}ric},
	Author ={ Grivaux, Sophie},
	Coden = {JFUAAW},
	Date-Added = {2010-08-20 18:58:50 +0300},
	Date-Modified = {2010-08-20 18:59:14 +0300},
	Doi = {10.1016/j.jfa.2005.06.001},
	Fjournal = {Journal of Functional Analysis},
	Issn = {0022-1236},
	Journal = {J. Funct. Anal.},
	Mrclass = {47A16 (47A10 47B33)},
	Mrnumber = {2159459 (2006i:47014)},
	Mrreviewer = {E. A. Gallardo-Guti{\'e}rrez},
	Number = {2},
	Pages = {281--300},
	Title = {Hypercyclicity and unimodular point spectrum},
	Url = {http://dx.doi.org/10.1016/j.jfa.2005.06.001},
	Volume = {226},
	Year = {2005},
	Bdsk-Url-1 = {http://dx.doi.org/10.1016/j.jfa.2005.06.001} }


	
	
	\bib{BaGr2}{article}{    
		author={Bayart, Fr{\'e}d{\'e}ric},    
		author={Grivaux, Sophie},    
		title={Frequently hypercyclic operators},    
		journal={Trans. Amer. Math. Soc.},    
		volume={358},    
		date={2006},    
		number={11},    
		pages={5083--5117 (electronic)},    
		issn={0002-9947},    
		review={\MR{2231886 (2007e:47013)}},    
		doi={10.1090/S0002-9947-06-04019-0}, } 	



\bib{BaGr3}{article}{
	Author = {Bayart, Fr{\'e}d{\'e}ric},
	Author = {Grivaux, Sophie},
	Date-Added = {2010-08-20 19:00:40 +0300},
	Date-Modified = {2010-08-20 19:01:00 +0300},
	Doi = {10.1112/plms/pdl013},
	Fjournal = {Proceedings of the London Mathematical Society. Third Series},
	Issn = {0024-6115},
	Journal = {Proc. Lond. Math. Soc. (3)},
	Mrclass = {47A16 (28C20 28D05 37A30 47A35)},
	Mrnumber = {2294994 (2008i:47019)},
	Mrreviewer = {Douglas P. Dokken},
	Number = {1},
	Pages = {181--210},
	Title = {Invariant {G}aussian measures for operators on {B}anach spaces and linear dynamics},
	Url = {http://dx.doi.org/10.1112/plms/pdl013},
	Volume = {94},
	Year = {2007},
	Bdsk-Url-1 = {http://dx.doi.org/10.1112/plms/pdl013}}
	

\bib{BM1}{article}{
	Author = {Bayart, Fr{\'e}d{\'e}ric},
	Author = {Matheron, \'Etienne},
	Coden = {JFUAAW},
	Date-Added = {2010-10-26 23:37:28 +0100},
	Date-Modified = {2010-10-26 23:37:44 +0100},
	Doi = {10.1016/j.jfa.2007.05.001},
	Fjournal = {Journal of Functional Analysis},
	Issn = {0022-1236},
	Journal = {J. Funct. Anal.},
	Mrclass = {47A16},
	Mrnumber = {2352487 (2008k:47016)},
	Mrreviewer = {Karl-Goswin Grosse-Erdmann},
	Number = {2},
	Pages = {426--441},
	Title = {Hypercyclic operators failing the hypercyclicity criterion on classical {B}anach spaces},
	Url = {http://dx.doi.org/10.1016/j.jfa.2007.05.001},
	Volume = {250},
	Year = {2007},
	Bdsk-Url-1 = {http://dx.doi.org/10.1016/j.jfa.2007.05.001}}


\bib{BM}{book}{
	Address = {Cambridge},
	Author = {Bayart, F.}
	Author = {Matheron, {\'E}.},
	Date-Added = {2010-03-19 13:52:18 +0000},
	Date-Modified = {2010-03-19 13:52:27 +0000},
	Doi = {10.1017/CBO9780511581113},
	Isbn = {978-0-521-51496-5},
	Mrclass = {47-02 (11M06 37B05 47A16 47A35 47Nxx)},
	Mrnumber = {MR2533318},
	Pages = {xiv+337},
	Publisher = {Cambridge University Press},
	Series = {Cambridge Tracts in Mathematics},
	Title = {Dynamics of linear operators},
	Url = {http://dx.doi.org/10.1017/CBO9780511581113},
	Volume = {179},
	Year = {2009},
	Bdsk-Url-1 = {http://dx.doi.org/10.1017/CBO9780511581113}}


	\bib{BELI}{article}{
				Author = {Bergelson, V.},
				Author = {Leibman, A.},
				Date-Added = {2010-12-07 15:41:04 +0000},
				Date-Modified = {2010-12-07 15:41:13 +0000},
				Doi = {10.1090/S0894-0347-96-00194-4},
				Fjournal = {Journal of the American Mathematical Society},
				Issn = {0894-0347},
				Journal = {J. Amer. Math. Soc.},
				Mrclass = {11B25 (05D10 28D05 54H20)},
				Mrnumber = {1325795 (96j:11013)},
				Mrreviewer = {Pierre Michel},
				Number = {3},
				Pages = {725--753},
				Title = {Polynomial extensions of van der {W}aerden's and {S}zemer\'edi's theorems},
				Url = {http://dx.doi.org/10.1090/S0894-0347-96-00194-4},
				Volume = {9},
				Year = {1996},
				Bdsk-Url-1 = {http://dx.doi.org/10.1090/S0894-0347-96-00194-4}}

\bib{Bernal}{article}{    
	author={Bernal-Gonz{\'a}lez, Luis},    
	title={Disjoint hypercyclic operators},    
	journal={Studia Math.},    
	volume={182},    
	date={2007},    
	number={2},    
	pages={113--131},    
	issn={0039-3223},    
	review={\MR{2338480 (2008k:47017)}},    
	doi={10.4064/sm182-2-2} }

\bib{BePe}{article}{
	Author = {B{\`e}s, J.},
	Author = {Peris, A.},
	Coden = {JFUAAW},
	Date-Added = {2010-12-05 18:32:13 +0000},
	Date-Modified = {2010-12-05 18:32:20 +0000},
	Doi = {10.1006/jfan.1999.3437},
	Fjournal = {Journal of Functional Analysis},
	Issn = {0022-1236},
	Journal = {J. Funct. Anal.},
	Mrclass = {47A16 (47B37)},
	Mrnumber = {1710637 (2000f:47012)},
	Mrreviewer = {Warren R. Wogen},
	Number = {1},
	Pages = {94--112},
	Title = {Hereditarily hypercyclic operators},
	Url = {http://dx.doi.org/10.1006/jfan.1999.3437},
	Volume = {167},
	Year = {1999},
	Bdsk-Url-1 = {http://dx.doi.org/10.1006/jfan.1999.3437}}


\bib{BePe2}{article}{
	Author = {B{\`e}s, Juan},
	Author = {Peris, Alfredo},
	Coden = {JMANAK},
	Date-Added = {2010-08-20 19:05:38 +0300},
	Date-Modified = {2010-08-20 19:06:47 +0300},
	Doi = {10.1016/j.jmaa.2007.02.043},
	Fjournal = {Journal of Mathematical Analysis and Applications},
	Issn = {0022-247X},
	Journal = {J. Math. Anal. Appl.},
	Mrclass = {47A16 (47B37 47B38)},
	Mrnumber = {2348507 (2008h:47020)},
	Mrreviewer = {Enhui Shi},
	Number = {1},
	Pages = {297--315},
	Title = {Disjointness in hypercyclicity},
	Url = {http://dx.doi.org/10.1016/j.jmaa.2007.02.043},
	Volume = {336},
	Year = {2007},
	Bdsk-Url-1 = {http://dx.doi.org/10.1016/j.jmaa.2007.02.043}}


	\bib{BG}{article}{    
		author={Bonilla, A.},    
		author={Grosse-Erdmann, K.-G.},    
		title={Frequently hypercyclic operators and vectors},    
		journal={Ergodic Theory Dynam. Systems},    
		volume={27},    
		date={2007},    
		number={2},    
		pages={383--404},   
		issn={0143-3857},    
		review={\MR{2308137 (2008c:47016)}},    
		doi={10.1017/S014338570600085X}} 	



		\bib{BonGroE}{article}{    
			author={Bonilla, A.},    
			author={Grosse-Erdmann, K.-G.},    
			title={Frequently hypercyclic operators and vectors---Erratum [MR    2308137]},    
			journal={Ergodic Theory Dynam. Systems},    
			volume={29},    
			date={2009},    
			number={6},    
			pages={1993--1994},    
			issn={0143-3857},    
			review={\MR{2563102 (2010k:47021)}},    
			doi={10.1017/S0143385709000959}, } 	
			
			

			\bib{BSH}{article}{    author={Bourdon, Paul S.},    author={Shapiro, Joel H.},    title={Hypercyclic operators that commute with the Bergman backward    shift},    journal={Trans. Amer. Math. Soc.},    volume={352},    date={2000},    number={11},    pages={5293--5316},    issn={0002-9947},    review={\MR{1778507 (2001i:47053)}},    doi={10.1090/S0002-9947-00-02648-9}, }


				\bib{ChSe2}{article}{    author={Chan, Kit C.},    author={Seceleanu, Irina},    title={Orbital limit points and hypercyclicity of operators on analytic    function spaces},    journal={Math. Proc. R. Ir. Acad.},    volume={110A},    date={2010},    number={1},    pages={99--109},    issn={1393-7197},    review={\MR{2666675 (2011d:47022)}} }

			\bib{ChSe}{article}{    
				author={Chan, Kit C.},    
				author={Seceleanu, Irina},    
				title={Hypercyclicity of shifts as a zero-one law of orbital limit points},    
				journal={J. Operator Theory},    
				volume={to appear},    
				date={2011},    
				number={},    
				pages={},   
				issn={},    
				review={},    
				doi={ }}

				\bib{CoSa}{article}{
				AUTHOR = {Costakis, George},
				Author = {Sambarino, Mart{\'{\i}}n},
				TITLE = {Topologically mixing hypercyclic operators},    
				JOURNAL = {Proc. Amer. Math. Soc.},   
				FJOURNAL = {Proceedings of the American Mathematical Society},     
				VOLUME = {132},       
				YEAR = {2004},    
				NUMBER = {2},      
				PAGES = {385--389},       
				ISSN = {0002-9939},      
				CODEN = {PAMYAR},    
				MRCLASS = {47A16 (37B05)},   
				MRNUMBER = {2022360 (2004i:47017)}, 
				MRREVIEWER = {Thomas Len Miller},        
				DOI = {10.1090/S0002-9939-03-07016-3},        
				URL = {http://dx.doi.org/10.1090/S0002-9939-03-07016-3}}

		\bib{CoRu}{article}{
			Author = {Costakis, G.},
			Author = {Ruzsa, I. Z.}
			Date-Added = {2010-10-26 18:07:16 +0100},
			Date-Modified = {2010-10-26 18:16:14 +0100},
			Journal = {preprint},
			Title = {Frequently Ces\`aro hypercylic operators are hypercyclic},
			Year = {2010}}
			
			

			\bib{DeRe}{article}{
				Author = {De la Rosa, Manuel},
				Author = {Read, Charles},
				Date-Added = {2010-10-26 23:40:27 +0100},
				Date-Modified = {2010-10-26 23:40:37 +0100},
				Fjournal = {Journal of Operator Theory},
				Issn = {0379-4024},
				Journal = {J. Operator Theory},
				Mrclass = {47A16 (47A15)},
				Mrnumber = {2501011 (2010e:47023)},
				Mrreviewer = {Karl-Goswin Grosse-Erdmann},
				Number = {2},
				Pages = {369--380},
				Title = {A hypercyclic operator whose direct sum $T\oplus T$ is not hypercyclic},
				Volume = {61},
				Year = {2009}}



				\bib{Fel}{article}{    author={Feldman, Nathan S.},    title={Hypercyclicity and supercyclicity for invertible bilateral    weighted shifts},    journal={Proc. Amer. Math. Soc.},    volume={131},    date={2003},    number={2},    pages={479--485},    issn={0002-9939},    review={\MR{1933339 (2003i:47032)}},    doi={10.1090/S0002-9939-02-06537-1}, }


		\bib{Fra}{article}{
								Author = {Frantzikinakis, N.},
								Date-Added = {2010-12-07 15:59:07 +0000},
								Date-Modified = {2010-12-07 15:59:30 +0000},
								Eprint = {0903.0042},
								Title = {Multiple recurrence and convergence for Hardy sequences of polynomial growth},
								Url = {http://arxiv.org/abs/0903.0042},
								Year = {2009},
								Bdsk-Url-1 = {http://arxiv.org/abs/0903.0042}}



\bib{FW}{article}{
				Author = {Frantzikinakis, N.},
				Author = { Wierdl, M.},
				Coden = {ADMTA4},
				Date-Added = {2010-12-07 15:57:44 +0000},
				Date-Modified = {2010-12-07 15:57:53 +0000},
				Doi = {10.1016/j.aim.2009.03.017},
				Fjournal = {Advances in Mathematics},
				Issn = {0001-8708},
				Journal = {Adv. Math.},
				Mrclass = {11B30 (05D10 11B25 28D05 37A45)},
				Mrnumber = {2531366 (2010f:11019)},
				Mrreviewer = {Bryna Kra},
				Number = {1},
				Pages = {1--43},
				Title = {A {H}ardy field extension of {S}zemer\'edi's theorem},
				Url = {http://dx.doi.org/10.1016/j.aim.2009.03.017},
				Volume = {222},
				Year = {2009},
				Bdsk-Url-1 = {http://dx.doi.org/10.1016/j.aim.2009.03.017}}

				\bib{Fur1}{article}{
					Author = {Furstenberg, H.},
					Date-Added = {2010-12-07 15:29:12 +0000},
					Date-Modified = {2010-12-07 15:30:08 +0000},
					Fjournal = {Journal d'Analyse Math\'ematique},
					Issn = {0021-7670},
					Journal = {J. Analyse Math.},
					Mrclass = {10L10 (10K10 28A65)},
					Mrnumber = {0498471 (58 \#16583)},
					Mrreviewer = {Francois Aribaud},
					Pages = {204--256},
					Title = {Ergodic behavior of diagonal measures and a theorem of {S}zemer\'edi on arithmetic progressions},
					Volume = {31},
					Year = {1977}}

\bib{Fur}{book}{
	Address = {Princeton, N.J.},
	Author = {Furstenberg, H.},
	Date-Added = {2010-11-16 17:59:57 +0000},
	Date-Modified = {2010-11-16 18:00:04 +0000},
	Isbn = {0-691-08269-3},
	Mrclass = {28D05 (10K10 10L10 54H20)},
	Mrnumber = {603625 (82j:28010)},
	Mrreviewer = {Michael Keane},
	Note = {M. B. Porter Lectures},
	Pages = {xi+203},
	Publisher = {Princeton University Press},
	Title = {Recurrence in ergodic theory and combinatorial number theory},
	Year = {1981}}
	

	
	
\bib{FurKa}{article}{
	Author = {Furstenberg, H.},
	Author = {Katznelson, Y.},
	Coden = {JOAMAV},
	Date-Added = {2010-12-07 15:20:52 +0000},
	Date-Modified = {2010-12-07 15:20:59 +0000},
	Fjournal = {Journal d'Analyse Math\'ematique},
	Issn = {0021-7670},
	Journal = {J. Analyse Math.},
	Mrclass = {28D05 (10K50 10L02 10L20)},
	Mrnumber = {531279 (82c:28032)},
	Mrreviewer = {J. H. B. Kemperman},
	Pages = {275--291 (1979)},
	Title = {An ergodic {S}zemer\'edi theorem for commuting transformations},
	Volume = {34},
	Year = {1978}}
	
	
\bib{GoSh}{article}{
Author = {Godefroy, G.},
Author = {Shapiro, J. H.},
Coden = {JFUAAW},
Date-Added = {2010-10-26 18:00:30 +0100},
Date-Modified = {2010-10-26 18:00:38 +0100},
Doi = {10.1016/0022-1236(91)90078-J},
Fjournal = {Journal of Functional Analysis},
Issn = {0022-1236},
Journal = {J. Funct. Anal.},
Mrclass = {47A65 (47B37 47B38 58F13)},
Mrnumber = {1111569 (92d:47029)},
Mrreviewer = {Warren R. Wogen},
Number = {2},
Pages = {229--269},
Title = {Operators with dense, invariant, cyclic vector manifolds},
Url = {http://dx.doi.org/10.1016/0022-1236(91)90078-J},
Volume = {98},
Year = {1991},
Bdsk-Url-1 = {http://dx.doi.org/10.1016/0022-1236(91)90078-J}}


\bib{Grivaux}{article}{    author={Grivaux, Sophie},    title={Hypercyclic operators, mixing operators, and the bounded steps    problem},    journal={J. Operator Theory},    volume={54},    date={2005},    number={1},    pages={147--168},    issn={0379-4024},    review={\MR{2168865 (2006k:47021)}} }

		

\bib{GrRo}{article}{
	Author = {Grivaux, S.},
	Author = {Roginskaya, M.},
	Date-Added = {2010-12-04 21:03:45 +0200},
	Date-Modified = {2010-12-04 21:03:55 +0200},
	Doi = {10.1112/blms/bdp084},
	Fjournal = {Bulletin of the London Mathematical Society},
	Issn = {0024-6093},
	Journal = {Bull. Lond. Math. Soc.},
	Mrclass = {47A16 (47B37)},
	Mrnumber = {2575335 (2011a:47025)},
	Mrreviewer = {E. A. Gallardo-Guti{\'e}rrez},
	Number = {6},
	Pages = {1041--1052},
	Title = {Multiplicity of direct sums of operators on {B}anach spaces},
	Url = {http://dx.doi.org/10.1112/blms/bdp084},
	Volume = {41},
	Year = {2009},
	Bdsk-Url-1 = {http://dx.doi.org/10.1112/blms/bdp084}}


	\bib{GP}{article}{    author={Grosse-Erdmann, K.-G.},    author={Peris, Alfredo},    title={Frequently dense orbits},    language={English, with English and French summaries},    journal={C. R. Math. Acad. Sci. Paris},    volume={341},    date={2005},    number={2},    pages={123--128},    issn={1631-073X},    review={\MR{2153969 (2006a:47017)}},    doi={10.1016/j.crma.2005.05.025} }
	
	
	\bib{GrPe}{book}{        
		author={Grosse-Erdmann, K.-G.},    
		author={Peris, Alfredo},
		title={Linear Chaos}, 
		Publisher = {Springer, Universitext, to appear},     
		volume={},    
		date={2011},    
		number={},    
		pages={},   
		issn={},    
		review={},    
		doi={}}



\bib{Leon}{article}{
	Author = {Le{\'o}n-Saavedra, Fernando},
	Date-Added = {2010-05-11 14:35:24 +0100},
	Date-Modified = {2010-05-11 14:35:56 +0100},
	Doi = {10.4064/sm152-3-1},
	Fjournal = {Studia Mathematica},
	Issn = {0039-3223},
	Journal = {Studia Math.},
	Mrclass = {47A16 (47A58 47B37)},
	Mrnumber = {MR1916224 (2003f:47012)},
	Mrreviewer = {H{\'e}ctor N. Salas},
	Number = {3},
	Pages = {201--215},
	Title = {Operators with hypercyclic {C}es\`aro means},
	Url = {http://dx.doi.org/10.4064/sm152-3-1},
	Volume = {152},
	Year = {2002},
	Bdsk-Url-1 = {http://dx.doi.org/10.4064/sm152-3-1}}

	\bib{Salas}{article}{
	    author={Salas, H{\'e}ctor N.},
	    title={Hypercyclic weighted shifts},
	    journal={Trans. Amer. Math. Soc.},
	    volume={347},    
	    date={1995},    
	    number={3},    
	    pages={993--1004},    
	    issn={0002-9947},    
	    review={\MR{1249890 (95e:47042)}},    
	    doi={10.2307/2154883} }

\bib{TV}{book}{
	Address = {Cambridge},
	Author = {Terence, T.},
	Author = {Vu, V. H.},
	Date-Added = {2010-12-07 15:23:00 +0000},
	Date-Modified = {2010-12-07 15:23:05 +0000},
	Isbn = {978-0-521-13656-3},
	Mrclass = {11-02 (05-02 05D10 11B13 11B30 11P70 11P82 28D05)},
	Mrnumber = {2573797},
	Note = {Paperback edition [of MR2289012]},
	Pages = {xviii+512},
	Publisher = {Cambridge University Press},
	Series = {Cambridge Studies in Advanced Mathematics},
	Title = {Additive combinatorics},
	Volume = {105},
	Year = {2010}}



\end{biblist}
\end{bibsection}
\end{document}